\documentclass{amsart}
\usepackage[mathscr]{eucal}
\usepackage{amsmath,amssymb,hyperref}
\usepackage{amstext}
\usepackage{graphics}
\usepackage{xcolor}
\usepackage{mathrsfs}
\usepackage{epstopdf}

\usepackage{tikz}
\usetikzlibrary{matrix}

\usepackage[all]{xy}

\usepackage[T1]{fontenc}
\usepackage[utf8]{inputenc}

\newtheorem{thm}{Theorem}[section]

\newtheorem{cor}[thm]{Corollary}
\newtheorem{prop}[thm]{Proposition}

\newtheorem{lemma}[thm]{Lemma}

\theoremstyle{definition}

\newtheorem{definition}[thm]{Definition}

\def\X0{X^{\circ}}

\def\Y0{Y^{\circ}}

\input xy
\xyoption{all}

\addtocounter{section}{0}             
\numberwithin{equation}{section}       

\begin{document}

\title[Positive rational nodal leaves on surfaces]
{Positive rational nodal leaves on surfaces}


\author[E. A. Santos]{Edileno de Almeida SANTOS}
\address{IMPA, Estrada Dona Castorina, 110, Horto, Rio de Janeiro,
Brasil}
\email{edileno@impa.br}


\thanks{The author is supported by CNPq}

\begin{abstract}
We consider singular holomorphic foliations on compact complex surfaces with invariant rational nodal curve of positive self-intersection. Then, under some assumptions, we list all possible foliations.
\end{abstract}

\maketitle

\setcounter{tocdepth}{1}
\sloppy

\section{Introduction}\label{S:Introduction}

 Let $X$ be a compact complex surface and $\mathcal F$ a  codimension one singular holomorphic foliation on it. This work aims at generalizing the following result of Brunella (see \cite{Brunella} and \cite{Brunella2}):


\begin{thm}\label{T:Brunella}
Let $\mathcal F$ be a foliation on a compact complex surface $X$ and let $C\subset X$ be a rational curve with a node $p$, invariant by $\mathcal F$, and with $C^2=3$. Suppose that $p$ is a reduced nondegenerate singularity of $\mathcal F$, and that it is the unique singularity of $\mathcal F$ on $C$. Then $\mathcal F$ is unique up to birational transformations.
\end{thm}

The unique foliation given by Theorem \ref{T:Brunella} will be called {\it Brunella's very special foliation} (see subsection \ref{S:C^2=3} for the definition).

But, what occurs if $C^2$ is an arbitrary positive integer? More specifically, we want to study/classify foliations on compact complex surfaces satisfying  assumptions similar to the ones of Theorem \ref{T:Brunella} with the hypothesis $C^2=3$ replaced by $C^2=n$, where $n$ is an arbitrary positive integer.

\begin{definition}\label{D:Link}
Let $\mathcal F$ be a foliation on a compact complex surface $X$. A {\it link} for $\mathcal F$ is a rational nodal curve $C\subset X$ with only one node $p\in C$ such that:
\begin{enumerate}
\item $C$ is {\it positive}, that is, $C^2=n>0$;
\item $C$ is $\mathcal F$-invariant;
\item $p$ is a reduced nondegenerate singularity of $\mathcal F$, and it is the unique singularity of $\mathcal F$ on $C$.
\end{enumerate}
\end{definition}

The existence of $C\subset X$, $C^2=n>0$, implies that $X$ is a projective surface (see \cite{BPV}, Theorem 6.2, page 160).


Our main purpose in this paper is  to prove the following theorem:

\begin{thm}\label{T:Main}
Let $\mathcal F$ be a foliation on a compact complex surface $X$ and let $C\subset X$ be a link for $\mathcal F$. Then we have only three possibilities, each one unique up to birational transformations:
\begin{enumerate}
\item $C^2=1$ and $\mathcal F$ is birational to a foliation $\mathcal F_1$ on $Bl_3(\mathbb{P}^2)/\alpha$, where $\alpha \in Aut(Bl_3(\mathbb{P}^{2}))$ and $Bl_3(\mathbb{P}^2)$ is a blow-up of $\mathbb{P}^2$ in three non-collinear points;
\item $C^2=2$ and $\mathcal F$ is birational to a foliation $\mathcal F_2$ on $\mathbb P^1\times \mathbb P^1/\beta$, $\beta \in Aut(\mathbb P^1\times \mathbb P^1)$;
\item $C^2=3$ and $\mathcal F$ is birational to a foliation $\mathcal F_3$ on $\mathbb P^2/\gamma$ (Brunella's very special foliation), $\gamma \in Aut(\mathbb P^{2})$.
\end{enumerate}
\end{thm}

\section{Some results in algebraic and complex geometry} For the reader's convenience, we summarize here some classical fundamentals results which will be used along this paper.

\subsection{Bimeromorphic geometry}

\begin{definition} [{\it Exceptional Curves}]
A compact, reduced, connected curve $C$ on a nonsingular surface $X$ is called {\it exceptional}, if there is a bimeromorphic map $\pi:X\rightarrow Y$ such that $C$ is exceptional for $\pi$, i.e., if there is an open neighbourhood $U$ of $C$ in $X$, a point $y\in Y$, and a neighbourhood $V$ of $y$ in $Y$, such that $\pi$ maps $U-C$ biholomorphically onto $V-\{y\}$, whereas $\pi(C)=y$. We shall express this situation also by saying that $C$ is {\it contracted} to $y$.
\end{definition}

\begin{thm}[Grauert's criterion, \cite{BPV}, page 91]\label{Grauert}
A reduced, compact connected curve $C$ with irreducible components $C_i$ on a smooth surface is exceptional if and only if the intersection matrix ($C_i\cdot C_j$) is negative definite.
\end{thm}

\begin{definition}[{\it Exceptional curves of the first kind}]
These are nonsingular rational curves with self-intersection $-1$. Frequently we call such curves $(-1)$-curves. A very useful characterisation of $(-1)$-curves is given by\end{definition}

\begin{thm}[\cite{BPV}, page 97]
Let $X$ be a nonsingular surface, $E \subset X$ a $(-1)$-curve and $\pi:X\rightarrow Y$ the map contracting $E$. Then $y = \pi (E)$ is nonsingular on $Y$.
\end{thm}

\begin{thm} [Uniqueness of the $\sigma$-process, \cite{BPV}, page 98]
Let $X$ and $Y$ be smooth surfaces and $\pi: X\rightarrow Y$ a bimeromorphic map. If $E = \pi^{-1}(y)$ is an irreducible curve, then near $E$, the map $\pi$ is equivalent to the $\sigma$-process with centre $y$. 
\end{thm}

\begin{lemma} [Factorization lemma, \cite{BPV}, page 98]\label{Factorization}
Let $\pi: X\rightarrow Y$ be a bimeromorphic map with $X$, $Y$ nonsingular surfaces. Unless it is an isomorphism, there is a factorization $\pi=\pi'\circ \sigma$, where $\sigma:X\rightarrow X$ is a $\sigma$-process.
\end{lemma}

\begin{cor} [Decomposition of bimeromorphic maps, \cite{BPV}, page 98]\label{Decomposition}
Let $X$, $Y$ be non-singular and $\pi: X\rightarrow Y$ a bimeromorphic map. Then $\pi$ is equivalent to a succession of $\sigma$-transforms, which locally (with respect to $Y$) are finite in number.
\end{cor}

\begin{thm} [\cite{BPV}, page 192]\label{T:Fibrations}
Let $X$ be a compact surface and $C$ a smooth rational curve on $X$. If $C^2 = 0$, then there exists a modification $\pi:X \rightarrow Y$, where $Y$ is ruled, such that $C$ meets no exceptional curve of $\pi$, and $\pi(C)$ is a fibre of $\pi$
\end{thm}

\subsection{Complex geometry} 
\begin{lemma}[\cite{Ueda}, Lemma 5]\label{Extension}
Let $X$ be a compact complex manifold of dimension $n>1$, $K$ a compact subset of $X$ and $E$ a holomorphic vector bundle over $X$. If $X$ is strongly pseudoconvex, then every section $s$ of $E$ over $X-K$ can be extended to a meromorphic section $\tilde{s}$ over all of $X$.
\end{lemma}
\begin{lemma}[\cite{Neeman}, page $32$]\label{Pseudoconvexity}
Let $X$ be a compact complex surface and $C\subset X$ a compact irreducible curve. If $C^2>0$ then $X-C$ is strongly pseudoconvex. 
\end{lemma}

\section{Existence}\label{S:Existence}
 For us a {\it cycle of smooth rational curves} (or simple a {\it cycle}) always means the union of a finite number of smooth rational curves in general position $C_i$, $i=1,..., m$, $m>1$, such that:
 if $m=2$, then $\# C_1 \cap C_2=2$;
 if $m>2$, then $\# C_i \cap C_{(i+1)}=\# C_1 \cap C_m=1$, $i=1,..., m-1$, otherwise $\# C_i \cap C_j=0$.
\subsection{Existence for $C^2=3$ (Brunella's very special foliation)} \label{S:C^2=3} Let $\mathcal L$ be the linear foliation on $\mathbb P^2$ given in affine coordinates by the linear $1$-form
\[
\omega=\lambda ydx-xdy=(\frac{1\pm\sqrt{-3}}{2}) ydx-xdy.
\]
This foliation has an invariant cycle of three lines $C_1\cup C_2 \cup C_3$. Moreover, the foliation $\mathcal L$ is $\gamma$-invariant, where $\gamma: (s:t:u)\longmapsto (u:s:t)$ is in $Aut(\mathbb P^2)$.

The quotient foliation $\mathcal F_3=\mathcal L /\gamma$ obtained by taking the quotient of $(\mathbb P^2, \mathcal L)$ by the group generated by $\gamma$ is, by definition, {\it Brunella's very special foliation}.

Note that the choose of $\lambda$ don't affect the birational class of $\mathcal F_3$, since the involution $(x,y)\mapsto (y,x)$ conjugates the two possible constructions.

\subsection{Existence for $C^2=2$} \label{S:C^2=2}
We take the foliation $\mathcal M$ on $\mathbb P^1\times \mathbb P^1$ given in affine coordinates $(x,y)$ by the linear 1-form
\[
\omega=\lambda ydx-xdy=\pm \sqrt{-1} ydx-xdy.
\]
where $\lambda =\pm \sqrt{-1}$. Then it leaves invariant the cycle of four lines 
\[
(\mathbb P^1\times \left\{0\right\})\cup (\mathbb P^1 \times \left\{\infty\right\})\cup (\left\{0\right\} \times \mathbb P^1)\cup (\left\{\infty\right\} \times \mathbb P^1),
\]
 in which the only singularities are the crossing points, each one reduced nondegenerate.
The automorphism of order 4
\[
\beta:(u:v,z:w)\mapsto (z:w,v:u).
\]
is such that, in affine coordinates $(x,y)$, $\beta (x,y)=(y,\frac{1}{x})$ and
\[
\beta^*\omega=\beta^*(\lambda ydx-xdy)=\lambda \frac{1}{x}dy-y(-\frac{1}{x^2})dx,
\]
hence, since $\lambda =\pm \sqrt{-1}$,
\[
\omega \wedge \beta^*\omega=(\lambda ydx-xdy) \wedge (\lambda \frac{1}{x}dy+\frac{y}{x^2}dx)=(\lambda ^2+1)\frac{y}{x}dx\wedge dy=0.
\]

Note that $\beta$ permutes cyclically the cycle of four lines
 \[
(\mathbb P^1\times \left\{0\right\})\cup (\mathbb P^1 \times \left\{\infty\right\})\cup (\left\{0\right\} \times \mathbb P^1)\cup (\left\{\infty\right\} \times \mathbb P^1).
\]
Then the quotient foliation $\mathcal F_2$ obtained by taking the quotient of $(\mathbb P^1\times \mathbb P^1, \mathcal M)$ by the group generated by $\beta$ is the desired foliation, that is, $\mathcal F_2$ has a link of self-intersection $2$.

Again the choose of $\lambda$ don't affect the birational class of $\mathcal F_2$, since the involution $(u:v,z:w)\mapsto (z:w,u:v)$ conjugates the two possible constructions.

\subsection{Existence for $C^2=1$} \label{S:C^2=1}
Let $\mathcal L$ and $\gamma$ as in subsection \ref{S:C^2=3}. Recall that $\mathcal L$ has a cycle of three invariant lines $C_1\cup C_2 \cup C_3$, where $C_i=\{[z_1:z_2:z_3]\in \mathbb P^2 \vert z_i=0 \}$, $i=1,2,3$. Consider the standard Cremona transformation $f:\mathbb P^2\dashrightarrow \mathbb P^2$, $f([z_1:z_2:z_3])=[z_2z_3:z_1z_3:z_1z_2]$. Note that $\mathcal L$ is $f$-invariant.  

If we blow-up the crossing points of the cycle of three $\mathcal L$-invariant projective lines $C_1\cup C_2 \cup C_3$, we obtain a birational morphism $\pi_3:Bl_3(\mathbb P^2)\rightarrow \mathbb P^2$  and a foliation $\mathcal N=\pi_3^*\mathcal L$ with an invariant cycle of six smooth rational $(-1)$-curves, say $\tilde C_1\cup \tilde C_2\cup \tilde C_3\cup C_4\cup C_5\cup C_6$, in which the singularities of $\mathcal N$ are only the crossing points (and they are reduced nondegenerate). Note that $\alpha=\pi_3^{-1}\circ f\circ \pi_3:Bl_3(\mathbb P^2)\rightarrow Bl_3(\mathbb P^2)$ becomes an automorphism of order six that preserves the foliation and permutes cyclically the cycle of six invariant rational curves.



 The quotient foliation $\mathcal F_1=\mathcal N /\alpha$ has a link of self-intersection 1, hence $\mathcal F_1$ is the desired foliation.

\section{Riccati Foliations}

We develop here the first tools to proof our main result.
 
Let $\mathcal F$ be a foliation on $X$ which is Riccati with respect to a fibration $\pi:X\rightarrow B$, where $B$ is a nonsingular curve. If $R$ is a regular fibre of $\pi$ which is $\mathcal F$-invariant, then (\cite [Chapter 4]{Brunella}): there are at most two singularities on $R$ and there exists coordinates $(x,y) \in D \times \mathbb{P}^1$ around $R$, where $D$ is a disc, such that the foliation is given by the 1-form
\[
\omega=(a(x)y^2+b(x)y+c(x))dx+h(x)dy.
\]

Let $q$ be a singularity for $\omega$. After a change in the $y$ coordinate, we can suppose $q=(0,0)$. Writing $h(x)=h_kx^k+...$, where $k>0$ and $h_k \neq 0$, we define the multiplicity of the fiber $R$ as $l(\mathcal F, R)=k$. We want to prove the following property of $\mathcal{F}$:

\begin{lemma}\label{L:Lemma Riccati}
The exceptional divisor of the reduction of singularities of $\mathcal F$ at $q=(0,0)$ is a chain of rational curves $L_1$,...,$L_n$ such that there is at most one non-invariant component, and if $L_i$ is such component then
\[
L_i\cap L_j \neq \oslash \Rightarrow Sing(\tilde{\mathcal F})\cap L_j=1-\delta_{ij}
\]
where $\tilde{\mathcal F}$ is the reduced foliation and $\delta_{ij}$ is the Kronecker's delta, that is, $\delta_{ii}=1$ and $\delta_{ij}=0$ if $i \neq j$.
\end{lemma}
\begin{proof}
If the linear part of $\omega$ at $q$ is non trivial, the result can be checked directly. We then suppose that the linear part at $q$ is trivial. Then $b(0)=c(0)=c'(0)=0$ and $l(\mathcal F, R)=k>1$. Since $Sing(\omega)\subset Sing(\mathcal F)$ has codimension two, we have $a(0)\neq 0$. Therefore $\omega$ has algebraic multiplicity two at $q$. Since $b(0)^2-4a(0)c(0)=0$, $q$ is the unique singularity of $\mathcal F$ in $R$. The blow-up at $q$ has on $R'\cap E'$ ($E'$ is the exceptional divisor and $R'$ is the strict transform of $R$) a singularity of the type $d(xy)=0$ and no more singularities on $R'$. If we collapse $R'$, then $E'$ becomes a new fibre $R_1$ of a new Riccati foliation $\mathcal F_1$. In this way, there may be at most two singularities on $R_1$, but now $l(\mathcal F_1, R_1)<l(\mathcal F, R)=k$.

Applying this procedure ({\it flipping of fibre}) a finite number of times, we obtain a foliation $\mathcal F_m$ and an invariant fibre $R_m$ such that a generating 1-form for the foliation has algebraic multiplicity one. That is, if $\omega$ is that 1-form, then
\[
\omega_m=(a_m(x)y^2+b_m(x)y+c_m(x))dx+h_m(x)dy.
\]
with $c_m(0)=h_m(0)=0$, but $b_m(0)\neq 0$ or $c_m'(0)\neq 0$ or $h_m'(0)\neq 0$. Now, if the singularity $(0,0)$ is dicritical, then the generating vector field for the foliation has two non zero linearly independent eigenvectors, and the exceptional divisor of the reduction of singularities $\tilde{\mathcal F_m}$ at $(0,0)$ is a chain of rational curves $L_1$,...,$L_n$, such that if $L_i$ is the (unique) non-invariant component and $L_i\cap L_j \neq \varnothing$ then $Sing(\tilde{\mathcal F_m})\cap L_j=1-\delta_{ij}$. Since we can come back by blow-ups at points not equal to the $(0,0)$ point of $\mathcal F_m$ to the blow-up of the original foliation at the original singular point $q=(0,0)$, the property is also true for the reduction at $q$ and then we conclude the proof.

\end{proof}



\begin{prop}\label{P:Riccati}
Let $\mathcal F$ be a foliation on a compact complex surface $X$. Let $C=C_1\cup ... \cup C_n$ be a cycle of $n$ invariant smooth rational curves, where $n>1$. Suppose that $C\cap Sing(\mathcal F)=\bigcup_{i\neq j} C_i \cap C_j$ are reduced non-degenerate singularities of $\mathcal F$. If $\mathcal F$ is Riccati with respect to a rational fibration $\pi:X\rightarrow B$, then every fibre of $\pi$ through a point of $C\cap Sing(\mathcal F)$ is completely supported on $C$.
\end{prop}
\begin{proof}
Let $p\in C\cap Sing(\mathcal F)$. If $R=\pi^{-1}(\pi(p))$ is the fibre through $p$, we can write
\[
R=C_{i_1}\cup...\cup C_{i_k}\cup E_1 \cup...\cup E_l
\]
where $i_1,..., i_k \in \{1,..., n\}$ and $E_1,..., E_l$ are smooth rational curves not in $\{C_1,..., C_n\}$, and, by Theorem \ref{T:Fibrations} (see \cite{BPV}, page 192), there is a birational transformation 
\[
\sigma=\sigma_m\circ...\circ\sigma_1:X\rightarrow Y
\]
where each $\sigma_i$, $i=1,...,m$, is a blow-up at a point $p_i$, such that $S=\sigma(R)$ is a regular fibre for the fibration $\rho=\pi\circ\sigma^{-1}$($\sigma$ is contraction of components of $R$).

Note that if we blow-up a regular point of a foliation, the exceptional divisor is invariant, with only one singularity on it, of type $xdy+ydx$. Therefore if $p_i$ is a regular point for the induced foliation $(\sigma_m\circ...\circ\sigma_i)_*\mathcal F$, then $(\sigma_m\circ...\circ\sigma_i)^{-1}(p_i)=D_1\cup...\cup D_r$ is $\mathcal F$-invariant and there exists $D_l$ (rational curve) such that $\#D_l\cap(D_1\cup...\cup \widehat{D_l}\cup...\cup D_r)=\# D_l\cap Sing(\mathcal F)=1$. Now, if $C\cap (\sigma_m\circ...\circ\sigma_i)^{-1}(p_i)\neq \varnothing$, then, since $(\sigma_m\circ...\circ\sigma_i)^{-1}(p_i)$ is connected and $\mathcal{F}$-invariant, we conclude that $(\sigma_m\circ...\circ\sigma_i)^{-1}(p_i)\subset C$, hence $D_l=C_{i_l}$, which result the contradiction $1=\# D_k\cap Sing(\mathcal F)=\# C_{i_l}\cap Sing(\mathcal F)=2$. Then, if we contract $(\sigma_m\circ...\circ\sigma_i)^{-1}(p_i)$, we don't affect the cycle $C$.

 So we can look at $\sigma$ as a reduction of singularities of $\sigma_*(\mathcal F)$ in $S$ and use Lemma \ref{L:Lemma Riccati} to conclude: if $p\in C_i\cap C_j$ then $C_i$ or $C_j$ is a component of $R$, otherwise we will have a non-invariant component of $R$ with singularity. 

If the set $\{E_1,...,E_l\}$ is not empty, since $R$ is connected, there exist $C_i$ and $E_j$ components of $R$ such that $C_i\cap E_j\neq \varnothing$. Then $E_j$ is not $\mathcal F$-invariant. But $C_i$ has two singularities, then by Lemma \ref{L:Lemma Riccati} $C_i$ cannot intersect $E_j$. Then we have $\{E_1,...,E_l\}=\varnothing$.

\end{proof}


\begin{definition}\label{D:Cycle}
Let $\mathcal F$ be a foliation on a compact complex surface $X$. A {\it $(k,l)$-cycle} for $\mathcal F$ is a cycle of $k>1$ smooth rational curves $C=C_1 \cup ... \cup C_k\subset X$ such that:
\begin{enumerate}
\item $C^2=n>0$;
\item $C_i^2=l$, $i=1,..., n$;
\item $C$ is $\mathcal F$-invariant;
\item $C\cap Sing(\mathcal F)=\bigcup_{i\neq j} C_i \cap C_j$ are reduced nondegenerate singularities of $\mathcal F$.
\end{enumerate}
\end{definition}

\begin{cor}\label{C:(k,l)-cycle}
Let $\mathcal F$ be a foliation on a compact complex surface $X$ and let $C=C_1 \cup ... \cup C_k\subset X$ be a $(k,l)$-cycle for $\mathcal F$. Then $(k,l)\in \{(2,-1),(3,-1),(3,1),(6,-1)\}\cup \{(2m,0)\mid m\in \mathbb N\}$.
\end{cor}
\begin{proof}
The proof is just an easy application of Proposition \ref{P:Riccati}, using suitable blow-ups at the crossing points of the cycle or blow-downs of exceptional curves.

Let $C=C_1 \cup ... \cup C_k\subset X$ be a $(k,l)$-cycle for a foliation $\mathcal F$ on $X$. We can suppose that $C_i \cap C_{i+1}=\{p_i\}$, $i=1,...,k-1$, and $C_k \cap C_1=\{p_k\}$, where the $k$ points $p_1,...,p_k$ are distinct.

If $l>0$, choose $z\in C$ a crossing point. After a suitable sequence of $l$ blow-ups beginning at $z$, we obtain a new cycle of rational curves
\[
\tilde{C}=E_{l} \cup ... \cup E_1 \cup D_1 \cup D_2 \cup ... \cup D_{k}
\]
where $D_1^2=0$, $E_1^2=-1$, $E_2^2=-2$,..., $E_{l}^2=-2$, $D_2^2=l$, $D_3^2=l$,..., $D_{k}^2=l-1$, and $D_1 \cap E_1=\{p\}$, $D_1 \cap D_2=\{q\}$. Then, the foliation $\mathcal F$ is Riccati with respect to a rational fibration that has $D_1$ as a regular fibre. By Proposition \ref{P:Riccati}, a fibre $R$ through a point not in $D_1$ must be supported on $\tilde C$, and such a fibre must be also disjoint from $D_1$, since $D_1$ is a fibre. That is, we must have $R\subset \tilde C - D_1\subset E_{l} \cup ... \cup E_1 \cup D_2 \cup ... \cup D_{k}$. Since $D_1 \cap E_1\neq\varnothing$ and  $D_1 \cap D_2\neq\varnothing$, $R\subset \tilde C - (D_1\cup E_1 \cup D_2)\subset E_{l} \cup ... \cup E_2 \cup D_3 \cup ... \cup D_{k}$. If $k=2$ and $l=1$, then, in fact, $R\subset \tilde C - (D_1\cup E_1 \cup D_2)=\varnothing$, and we obtain a contradiction, since $R$ cannot be empty. For $k>2$ or $l>1$, every connected curve supported on $E_{l} \cup ... \cup E_2 \cup D_3 \cup ... \cup D_{k}$ cannot be contracted to a rational curve of zero self-intersection, hence cannot be a fibre of a rational fibration. Therefore, there is no $(k,l)$-cycle if $l>0$.


Now, suppose $l=0$. Then, since $C_i^2=0$, $i=1,... , k$, we don't need take blow-ups to produce rational fibrations. Just choose, for example, $C_1$ as the fibre $R_1$ of a rational fibration and $\mathcal F$ Riccati with respect to this fibration. Suppose that $k=2m+1$ is odd. Take the fibre $R_2$ trough the crossing point $p_3$. Since $R_2$ must be supported on $C$, we obtain $R_2=C_3$. By the same reason, the fibre $R_3$ through the crossing point $p_5$ is $R_3=C_5$. Inductively, we obtain that the fibre $R_i$ through $p_{2i-1}$ is $R_i=C_{2i-1}$. Then $R_{m+1}=C_{2m+1}=C_k$ is the fibre through $p_{2m+1}=p_k$, which is impossible since the fibre through $p_k=p_{2m+1}$ is just $C_1\neq C_k$. Hence, if $l=0$, then $k$ must be even.

Finally, using contractions instead of blow-ups, we can conclude that there is no $(k,-1)$-cycle if $(k,-1)$ is not in $\{(2,-1),(3,-1),(6,-1)\}$.

\end{proof}

We can now give here a different proof of \cite [Chapter 3, Proposition 4]{Brunella}.

\begin{prop}\label{P:not Riccati}
Let $\mathcal F$ be one of the foliations $\mathcal F_1$, $\mathcal F_2$ or $\mathcal F_3$. Then $\mathcal F$ is not birrational to a Riccati foliation. 
\end{prop}
\begin{proof} Just like before, after one blow-up at the nodal point in the link of $\mathcal F$, we conclude, by Proposition \ref{P:Riccati}, that $\mathcal F$ cannot be Riccati.

\end{proof}

\section{Proof of the Theorem \ref{T:Main} }\label{S:Proof}

\subsection{Preliminary computations} Let $p$ be the node of $C$ and $C^2=n$ a positive integer. If the hypotheses for the foliation are as in the Introduction \ref{S:Introduction} (that is, $C$ is a link for $\mathcal F$), we can use the Camacho-Sad formula to calculate the quotient of eigenvalues of $\mathcal F$ at $p$ (see \cite[Chapter 3]{Brunella}):

\[
n=C^2=CS(\mathcal F, Y, p)=\lambda+2+\frac{1}{\lambda} .
\]

Then we have the equation

\[
\lambda^2+(2-n)\lambda+1=0
\]
whose solution is

\[
\lambda=\frac{n-2\pm\sqrt{n(n-4)}}{2} .
\]
 
Therefore:
\begin{enumerate}
\item if $C^2=1$ then $-\lambda$ is a $6^{th}$ primitive root of unit;
\item if $C^2=2$ then $-\lambda$ is a $4^{th}$ primitive root of unit;
\item if $C^2=3$ then $-\lambda$ is a $3^{th}$ primitive root of unit;
\item if $C^2=4$ then $\lambda=1$;
\item if $C^2>4$ then $\lambda$ is a positive irrational number.
\end{enumerate}
 
\subsection{Basic lemmas and propositions} Here we will develop some more "technology" for the proof of our main result.

The next lemma is the generalization of \cite [Chapter 3, Lemma 1]{Brunella}. The proof is essentially the same.

\begin{lemma}
Let $\mathcal F$ be a foliation on a compact complex surface $X$ and let $C\subset X$ be a link for $\mathcal F$ with node $p\in C$. Let $L=N_{\mathcal{F}}^{*}\otimes\mathcal{O}_{X}(C)$ and $\lambda$ be the quotient of eigenvalues at $p$. Suppose that $-\lambda$ is a $k^{th}$ primitive root of unit, $k>2$. Then there exists a neighbourhood $U$ of $C$ such that ${L^{\otimes k}}_{\mid {U}}=\mathcal{O}_U$.
\end{lemma}
\begin{proof}
Since $\lambda$ is non-real, given a point $q\in C-\{p\}$ and a transversal $T$ to $\mathcal F$ at $q$, the corresponding holonomy group of $\mathcal F$, $Hol_{\mathcal F}\subset \mathrm{Diff}(T, q)$, is infinite cyclic, generated by an hyperbolic diffeomorphism with linear part $exp(2\pi i\lambda)$ (\cite{CS} or \cite{MM}). Hence, there exists on $T$ a $Hol_{\mathcal F}$-linearising coordinate $z$, $z(q)=0$. We extend this coordinate to a full neighbourhood of $q$ in $X$, constantly on the local leaves of $\mathcal F$. The logarithmic 1-form $\eta_q=\frac{dz}{z}$ defines $\mathcal F$, is closed, and $\eta_q\mid_T$ is $Hol_{\mathcal F}$-invariant.

By the Poincaré linearisation theorem, in a neighbourhood of $p$ the foliation is defined by a closed logarithmic  1-form $\eta_p=\frac{dz}{z}-\lambda\frac{dw}{w}$ (\cite{CS} or \cite{MM}). If $q$ is close to $p$, then $\eta_p\mid_T$ is $Hol_{\mathcal F}$-invariant.

We obtain a neighbourhood $U$ of $C$ by the union of the open sets $U_j$, such that in each $U_j$ the foliation is defined by a logarithmic 1-form $\eta_j$, with poles on $C$, which is closed and $Hol_{\mathcal F}$-invariant at the transversals. On $U_i\cap U_j$ we have $\eta_i=f_{ij}\eta_j$, $f_{ij}\in \mathcal{O}^*$. The closedness of $\eta_i$ and $\eta_j$ implies that $df_{ij}\wedge \eta_j=0$, then $f_{ij}$ is constant along the local leaves of $\mathcal F$. Moreover, $f_{ij}\mid_T$ is $Hol_{\mathcal F}$-invariant and hence constant because the holonomy is hyperbolic.

Thinking ${\eta_j}$ as local sections of $L=N_{\mathcal{F}}^{*}\otimes\mathcal{O}_{X}(C)$, then the previous property shows that $L_{\mid U}$ is defined by a locally constant cocycle. Hence, to show that ${L^{\otimes k}}_{\mid U}=\mathcal{O}_U$ it is sufficient to show that ${L^{\otimes k}}_{\mid C}=\mathcal{O}_C$. We can now use the residue of $\eta_j$ along $C$ to calculate the cocycle. For $\eta_q$ with $q\in C-\{p\}$ we can choose the 1-form to produce any non-zero residue. But we have a restriction around $p$: the residue of $\eta_p$ on one separatrix is $-\lambda$ times the residue on the other separatrix. Since $(-\lambda)^k=1$, its is clear that ${L^{\otimes k}}_{\mid C}=\mathcal{O}_C$.

\end{proof}

Also the next proposition is an easy adaptation of Brunella's argument in \cite [Chapter 3, page 61-62]{Brunella}.

\begin{prop}\label{P:Covering}
Let $\mathcal F$ be a foliation on a compact complex surface $X$ and let $C\subset X$ be a link for $\mathcal F$ with node $p\in C$. Let $\lambda$ be the quotient of eigenvalues at $p$. Suppose that $-\lambda$ is a $k^{th}$ primitive root of unit, $k>2$. Then there exists a compact surface $Z$, a transformation $f:Z\rightarrow X$, a neighbourhood $U$ of $C$ and an open set $V\subset Z$  such that  $f\mid_{V}:V\rightarrow U$ is a regular $k$-covering over $U$. Moreover, $f\mid_{V}^{-1}(C)$ is a cycle of $k$ smooth rational curves, each one with self-intersection $C^2-2$ (that is, a $(k,C^2-2)$-cycle), and the deck transformations of $f\mid_{V}$ permutes cyclically the curves in the cycle.
\end{prop}
\begin{proof}
By the above lemma, the line bundle $L^{\otimes k}$ has a nontrivial section over $U$ without zeroes. Since $C^2>0$, the open set $X-C$ is strictly pseudoconvex by Lemma \ref{Pseudoconvexity}. Then, by Lemma \ref{Extension} , that section can be extended to the full $X$ as a global meromorphic section $s$ of $L^{\otimes k}$. Consider $E(L^{\otimes k})$ the compactification of the total space of $L^{\otimes k}$. Let $\tilde{s}$ the compactification of the graph of $s$ in $E(L^{\otimes k})$. Let $\tau:E(L)\rightarrow E(L^{\otimes k})$ be the map defined by the $k^{th}$ tensor power.

 Let $Z$ be the desingularisation of $\tau^{-1}(\tilde{s})$ and elimination of indeterminacies of the projection $\tau^{-1}(\tilde{s})\dashrightarrow X$. Take $f:Z\rightarrow X$ the induced projection.

\end{proof}

\begin{lemma}\label{L:Conjugation}
Let $p_1=(1:0:0)$, $p_2=(0:1:0)$, $p_3=(0:0:1)$ be three non collinear points in $\mathbb P^2$. Let $\gamma\in Aut(\mathbb P^2)$ given by $\gamma(z_1:z_2:z_3)=(z_3:z_1:z_2)$. If $J\in Aut(\mathbb P^2)$ is another automorphism such that $J(p_1)=p_2$, $J(p_2)=p_3$ and $J(p_3)=p_1$, then $J$ is conjugated to $\gamma$, that is, there is $g\in Aut(\mathbb P^2)$ such that $\gamma=g\circ J\circ g^{-1}$.
\end{lemma}
\begin{proof}
In homogeneous coordinates, $J(z_1:z_2:z_3)=(x z_3:y z_1:z z_2)$, where $xyz \neq 0$. Note that we can suppose $xyz = 1$. Since $Aut(\mathbb P^2)=PGL(3,\mathbb C)$, writing $J$ and $\gamma$ as matrices, $J=\left(
\begin{array}{ccc}
0& 0& x\\
y& 0& 0\\
0& z& 0\\
\end{array}
\right)
$ and $\gamma=\left(
\begin{array}{ccc}
0& 0& 1\\
1& 0& 0\\
0& 1& 0\\
\end{array}
\right)
$, we need to show that there is a matrix 
$ A=\left(
\begin{array}{ccc}
a_1& a_2& a_3\\
b_1& b_2& b_3\\
c_1& c_2& c_3\\
\end{array}
\right)
\in GL(3, \mathbb C)$, such that $AJ=\gamma A$ in $PGL(3,\mathbb C)$.

If $a=(a_1, a_2, a_3)$, $b=(b_1, b_2, b_3)$, $c=(c_1, c_2, c_3)$, it's easy to see that the equality $AJ=\gamma A$ is equivalent to
$x \gamma(c)=a$, $y \gamma(a)=b$, $z \gamma(b)=c$. Take $a\neq 0$ and define $b=y \gamma(a)$ and $c=z \gamma(b)=z y \gamma^2(a)$. Then the matrix $A=(a, b, c)\in GL(3, \mathbb C)$ is a solution.

\end{proof}

\begin{prop}\label{P:3-cycle}
Let $\mathcal F$ be a foliation on a compact complex surface $Z$ and let $C_1\cup C_2\cup C_3\subset Z$ be a $(3,1)$-cycle for $\mathcal F$. Suppose that there exists a birational $\mathcal F$-automorphism  $\phi:Z\dashrightarrow Z$ of order three permuting cyclically the rational curves. Then $\mathcal F$ is birational to the linear foliation $\mathcal L$ on $\mathbb P^2$ from subsection \ref{S:C^2=3} and the quotient foliation $\mathcal {F}/\phi$ is birational to  $\mathcal{F}_3=\mathcal {L}/\gamma$.
\end{prop}
\begin{proof}
We can suppose $\phi(C_1)=C_2$, $\phi(C_2)=C_3$ and $\phi(C_3)=C_1$. Take, for each $i$, a section $s_i$ of $\mathcal O_Z(C_i)$ vanishing on $C_i$. Since $C_1$, $C_2$, $C_3$ are linearly equivalent, we can define a rational map
\[
(s_1:s_2:s_3):Z----\rightarrow \mathbb P^2.
\]
It's easy to see that this map is birational and biregular in a neighbourhood of the cycle $C_1\cup C_2\cup C_3$, whose image is a cycle of three lines in $\mathbb P^2$. The induced foliation $\tilde{\mathcal{F}}$ on $\mathbb P^2$ is linear because the degree of the foliation is 1. The birational automorphism $\phi$ is mapped to a birational automorphism $\tilde \phi$ of $\mathbb P^2$ which is biregular in a neighbourhood of the three lines and hence everywhere; moreover these automorphism permutes cyclically the three lines. By Lemma \ref{L:Conjugation} $\tilde \phi$ is conjugated to the automorphism $\gamma(z_1:z_2:z_3)=(z_3:z_1:z_2)$, that is, there is $g\in Aut(\mathbb P^2)$ such that $\gamma=g\circ \tilde{\phi} \circ g^{-1}$. Since $\gamma$ is an $g_*\tilde{\mathcal{F}}$-automorphism, an easy computation shows that $g_*\tilde{\mathcal{F}}=\mathcal L$ in homogeneous coordinates $[z_1:z_2:z_3]$. In particular, $\mathcal {F}/\phi$ is birational to  $\mathcal{F}_3=\mathcal {L}/\gamma$.

\end{proof}

Analogously we can prove the following two results.
\begin{lemma}\label{L:Conjugation2}
Let $p_1=(1:0,1:0)$, $p_2=(0:1,1:0)$, $p_3=(0:1,0:1)$, $p_4=(1:0,0:1)$  be four points in $\mathbb P^1 \times \mathbb P^1$. Let $\beta\in Aut(\mathbb P^1 \times \mathbb P^1)$ given by $\beta(z_1:z_2,z_3:z_4)=(z_4:z_3,z_1:z_2)$. If $J\in Aut(\mathbb P^1 \times \mathbb P^1)$ is another automorphism such that $J(p_1)=p_2$, $J(p_2)=p_3$, $J(p_3)=p_4$ and $J(p_4)=p_1$, then $J$ is conjugated to $\beta$, that is, there is $g\in Aut(\mathbb P^1 \times \mathbb P^1)$ such that $\beta=g\circ J\circ g^{-1}$.
\end{lemma}
\qed

\begin{prop}\label{P:4-cycle}
Let $\mathcal H$ be a foliation on a compact complex surface $W$ and let $D_1\cup D_2\cup D_3\cup D_4\subset W$ be a $(4,0)$-cycle for $\mathcal H$. Suppose that there exists a birational $\mathcal H$-automorphism  $\phi:W\dashrightarrow W$ of order four permuting cyclically the rational curves. Then $\mathcal H$ is birational to the linear foliation $\mathcal M$ on $\mathbb P^1 \times \mathbb P^1$ from subsection \ref{S:C^2=2} and the quotient foliation $\mathcal {W}/\phi$ is birational to  $\mathcal{F}_2=\mathcal {M}/\beta$.
\end{prop}
\begin{proof}
Take, for every $i$, a section $s_i$ of $\mathcal O_Z(D_i)$ vanishing on $D_i$. We define a rational map
\[
(s_1:s_2,s_3:s_4):W----\rightarrow \mathbb P^1\times \mathbb P^1.
\]

It's easy to see that this map is birational and biregular in a neighbourhood of the cycle $D_1\cup D_2\cup D_3\cup D_4$, whose image is a cycle of four lines in $\mathbb P^1\times \mathbb P^1$. Therefore, the induced foliation $\tilde{\mathcal{H}}$ on $\mathbb P^1\times \mathbb P^1$ leaves invariant the cycle of four lines 
\[
(\mathbb P^1\times \left\{0\right\})\cup (\mathbb P^1 \times \left\{\infty\right\})\cup (\left\{0\right\} \times \mathbb P^1)\cup (\left\{\infty\right\} \times \mathbb P^1)
\]
whose singularities on the cycle are only the crossing points, each one reduced nondegenerate. According to \cite[Chapter 4, Proposition 1]{Brunella} (see also \cite{Lins-Neto} and \cite{Lins-Neto2}) we have that this foliation on $\mathbb P^1\times \mathbb P^1$ is given in affine coordinates $(x,y)$ by a linear 1-form
\[
\omega=\lambda ydx-xdy.
\]

The birational automorphism $\phi$ is mapped to a birational automorphism $\tilde \phi$ of $\mathbb P^1\times \mathbb P^1$ which is biregular in a neighbourhood of the four lines and hence everywhere; moreover these automorphism permutes cyclically the four lines. By Lemma \ref{L:Conjugation2} $\tilde \phi$ is conjugated to the automorphism $\beta(z_1:z_2,z_3:z_4)=(z_3:z_4,z_2:z_1)$, that is, there is $g\in Aut(\mathbb P^2)$ such that $\beta=g\circ \tilde{\phi} \circ g^{-1}$. Since $\beta$ is an $g_*\tilde{\mathcal{H}}$-automorphism, an easy computation shows that $g_*\tilde{\mathcal{H}}=\mathcal M$ in homogeneous coordinates $[z_1:z_2,z_3:z_4]$. In particular, $\mathcal {H}/\phi$ is birational to  $\mathcal{F}_2=\mathcal {M}/\beta$.

\end{proof}

Now we are read to finish the proof of the theorem.

\subsection{Self-intersection 1}
 Since $-\lambda$ is a $6^{th}$ primitive root of unit, by Proposition \ref{P:Covering} we obtain a covering $F:Z\longrightarrow X$, regular and of order six in a neighbourhood $U$ of $C$. The deck transformations over $U$ extend, by construction, to birational transformations of $Z$. Let $\alpha:Z\dashrightarrow Z$ be the extended deck transformation of order six.
 
Now, we lift $\mathcal F$ to $Z$ via $F$, obtaining a new foliation $\mathcal G$ which leaves invariant six smooth rational curves $C_i$, $i=1$,..., $6$, forming a cycle over $C$. We have $C_i^2=-1$, because $C^2=1$. The only singularities of $\mathcal G$ at the cycle are the six crossing points, all reduced nondegererate as well as $p$.

We can contract three disjoint ($-1$)-curves of the cycle, say $C_1, C_3$ and $C_5$, obtaining a foliation $(\tilde{\mathcal G},\tilde{Z})$ birational to $(\mathcal G, Z)$. Note that $\tilde{\mathcal G}$ has an invariant cycle of three smooth rational curves with self-intersection 1. Furthermore, $\alpha^2=\alpha\circ \alpha$ induces a  birational $\tilde{\mathcal G}$-automorphism that permutes cyclically this cycle. Therefore, by Proposition \ref{P:3-cycle},  $\tilde{\mathcal G}$ is birational to the linear foliation $\mathcal{L}$ on $\mathbb P^2$ given in subsection \ref{S:C^2=3}. In the same way, contracting the three disjoint ($-1$)-curves $C_2, C_4$ and $C_6$, we also obtain a foliation birational to $(\mathcal{L},\mathbb P^2)$. Then $\alpha:Z\dashrightarrow Z$ induces a $\mathcal L$-automorphism $\tilde{\alpha}:\mathbb P^2\rightarrow \mathbb P^2$. Since $\tilde{\alpha}$ is unique up to conjugation (Lemma \ref{P:3-cycle}), the same is true for $\alpha$ . Therefore $\mathcal F$ is birational to the foliation $\mathcal F_1$ from subsection \ref{S:C^2=1}.

\subsection{Self-intersection 2}
In this case, $-\lambda$ is a $4^{th}$ primitive root of unit. By Lemma \ref{P:Covering} we have a covering $G:W\longrightarrow X$, which is regular and of order 4 on a neighbourhood of $C$. Lifting $\mathcal F$ to $W$, we obtain a foliation $\mathcal H$ which leaves invariant four smooth rational curves $D_i$, $i=1$,..., $4$, forming a cycle over $C$. Analogously, $D_i^2=0$, because $C^2=1$. The only singularities of $\mathcal H$ at the cycle are the four crossing points, all reduced nondegererate as well as $p$. Hence Proposition \ref{P:4-cycle} implies that $\mathcal F$ is birational to $\mathcal F_2$.



\subsection{Self-intersection 3}
This case is covered by Theorem \ref{T:Brunella}. Anyway, the proof is just Lemma \ref{P:Covering} plus Proposition \ref{P:3-cycle}.

\subsection{Self-intersection 4}
In this case, $\lambda=1$, therefore $p$ is a dicritical linerizable singularity (in particular, after a blow-up at $p$, the self-intersection of the strict transform of $C$ is $C^2-4=0$, so we obtain a rational fibration over $\mathbb P^1$) by \cite{CS} or \cite{MM}. But, since $\lambda$ is rational positive, the foliation is not reduced nondegenerate at $p$, hence this case is not possible in our assumptions.

\subsection{Self-intersection greater than 4}
Since $k>4$ we have that $\lambda$ is a positive irrational number, hence the singularity is non-dicritical linerizable.

After $k$ suitable blow-ups the self-intersection of the strict transform of $C$ is $\tilde{C}^2=C^2-4-k+1=n-3-k$ (the first blow-up at $p$ and the following blow-ups at one of the two singular points of the foliation in the strict transform of $C$). Therefore, after $n-3$ blow-ups we obtain $\tilde{C}^2=0$. Let $\sigma :\tilde{X}\longrightarrow X$ be the transformation obtained by composing theses blow-ups, $\tilde{C}=\sigma^{*}(C)$, $E=\sigma^{-1}(p)=C_1+...+C_{(n-3)}$, where the $C_i$ are rational curves, with $C_{1}^2=-1$ and $C_j^2=-2$ if $j>1$, and $\tilde{\mathcal F}=\sigma^{*}(\mathcal F)$. Since $\mathcal Z(\tilde{\mathcal F}, \tilde{C})=2$, $\tilde{\mathcal F}$ is a Riccati foliation with respect to a fibration $\pi:\tilde{X}\longrightarrow B$, where $B$ is a smooth curve (by \cite[Chapter 4, Proposition 1]{Brunella}). We can suppose that the fibration has connected fibres. Since the exceptional divisor $E$ is a union of smooth rational curves, the base $B$ is a smooth rational curve.

Let $q=C_1\cap C_2$, which is a singularity of the foliation,  and $R$ the fibre (possibly singular) through $q$. By Proposition \ref{P:Riccati}, $R$ must by supported on $E$, which is impossible, since $E$ has negative definite matrix of intersection.

\qed

\bibliographystyle{amsplain}

\end{document}